\documentclass{amsart}

\usepackage[T1]{fontenc}
\usepackage[utf8]{inputenc}
\usepackage{amsmath,amsthm,amsfonts,amscd,amssymb,eucal,latexsym,mathrsfs}
\usepackage{stmaryrd}
\usepackage{enumerate}
\usepackage{hyperref}
\usepackage[all]{xy}
\usepackage{etoolbox}

\usepackage{tikz,tikz-cd}
\usetikzlibrary{shapes.geometric}
\usetikzlibrary{arrows}
\usepackage{tqft}

\usetikzlibrary{positioning}
\usepackage{float}
\usepackage{MnSymbol}
\usetikzlibrary{matrix}

\theoremstyle{plain}
\newtheorem{theorem}{Theorem}[section]
\newtheorem*{theorem*}{Theorem}

\newtheorem*{corollary*}{Corollary}

\newtheorem{lemma}[theorem]{Lemma}

\theoremstyle{definition}

\newtheorem*{definition*}{Definition}

\usetikzlibrary{calc,graphs}
\usepackage{xcolor}
\usetikzlibrary{arrows,decorations.pathmorphing}

\newcommand{\ZI}{\mathbb{Z}}

\renewcommand{\1}{\mathbf{1}}
\newcommand{\0}{\mathbf{0}}

\DeclareMathOperator{\Aut}{\mathrm{Aut}}

\DeclareMathOperator{\inj}{\hookrightarrow}

\newcommand{\SU}{\mathrm{SU}}

\newcommand{\ts}{\textsection}

\newcommand{\ip}[1]{\langle#1\rangle} 

\title{Pauli matrices and ring puzzles}

\author{Sylvain Barr\'e}
\author{Mika\"el Pichot}
\address{Sylvain Barr\'e, UMR 6205, LMBA, Université de Bretagne-Sud,BP 573, 56017, Vannes, France}\email{Sylvain.Barre@univ-ubs.fr}
\address{Mika\"el Pichot, McGill University, 805 Sherbrooke St W., Montr\'eal, QC H3A 0B9, Canada}\email{pichot@math.mcgill.ca}

\begin{document}
\begin{abstract}
We study a family of tessellations of the Euclidean plane which are obtained by local developments of algebraic equations satisfied by the Pauli matrices.
\end{abstract}

\maketitle

The Pauli matrices are Hermitian matrices, often denoted 
\[
X := \begin{pmatrix}0&1\\1&0\end{pmatrix},\ \   Y := \begin{pmatrix}0&-i\\i&0\end{pmatrix},\ \   \text{ and } \ \ Z := \begin{pmatrix}1&0\\0&-1\end{pmatrix},
\]
introduced by Pauli to describe the electron spin (and, more generally, any  spin $\frac 1 2$ particle or qubit); thus, the commutation relations
\[
XY-YX=2iZ
\]
\[
YZ-ZY=2iX
\]
\[
ZX-XZ=2iY
\]
and  normalization 
\[
X^2+Y^2+Z^2=3
\] 
express in quantum physics the basic properties of the spin angular momentums for spin $\frac 1 2$ particles.

When viewed as elements of $\SU(2)$, the Pauli matrices  generate a finite subgroup $P=\ip{X,Y,Z}$ called the Pauli group. The algebraic structure of this group is well--known and elementary to describe (see \ts \ref{S - triangle}).

In this paper, we are interested the following lesser known set of three equations also satisfied by the Pauli matrices: 
\[
XYXZYZ=1
\]
\[
YZYXZX=1
\]
\[
ZXZYXY=1
\]
We are interested in planar tessellations associated with $X$, $Y$ and $Z$; roughly speaking, we wish to study how these equations fit together to ``fill in Euclidean planes of Pauli matrices''.

To put the problem in more precise terms, it is convenient to use the notion of a ring puzzle considered in \cite{autf2puzzles}. For the purpose of the present paper,  the following particular case is sufficient: a ring puzzle is, by definition,  a Euclidean plane, say $E$, tessellated by equilateral triangles, in such a way that the algebraic relations between the Pauli matrices are satisfied locally as follows: we imagine having as puzzle pieces  small triangles marked with Pauli matrices, and require that these pieces fit together at a vertex of the tessellation, when the cyclic product in $\SU(2)$ of the corresponding Pauli matrices is trivial. We require that this condition holds at every vertex. The question then is: what puzzles can be obtained which satisfy these constraints (and additional ones of nontriviality, see \ts\ref{S - Puzzle set}).

We refer to these tessellations  as Pauli puzzles. We prove that there are four families of Pauli puzzles: 
\[
\left.
    \begin{array}{l} \text{the $X$-puzzles}\\
\text{the $Y$-puzzles}\\
\text{the $Z$-puzzles}\\
 \end{array}
  \right\} = \text{three sets of two puzzles} 
\]
and an infinite family of pairwise nonisomorphic puzzles, which we call the $T$-puzzles. 
This result is the content of Theorem  \ref{T - Pauli classification 2}. 

The puzzles will be described explicitly in the course of the proof of Theorem  \ref{T - Pauli classification 2}. 

We now turn to the motivation of the paper, which lies in geometric group theory.  We refer to \cite{BH} for a general reference and definitions of  the concepts that are used below.

Let us first discuss \cite{autf2puzzles} briefly. 
Let $F_2$ denote the free group on 2 generators. The paper  \cite{autf2puzzles}  introduces a notion of ring puzzle associated with the group $\Aut(F_2)$ of automorphisms of $F_2$. This group acts properly discontinuously on a CAT(0) space with compact quotient (the Brady complex, see \cite{brady1994automatic,brady2000artin}) in which the flat planes define $\Aut(F_2)$-puzzles. It is proved in \cite{autf2puzzles} that the ring puzzles associated with $\Aut(F_2)$ can be classified up to isomorphism, namely, up to an isometric cell bijection preserving the cell structure of the Brady complex (see \cite[\ts 2]{autf2puzzles}).  

In \ts\ref{S - triangle} of the present paper, we construct a group $G$ from the Pauli matrices which  is an analog of $\Aut(F_2)$ in this context. The group $G$ also acts on a CAT(0) space $\Delta$. It is constructed from a triangle of group (defined in \ts\ref{S - triangle}) which is shown to be developable. (The group $G$ is an example of a group of intermediate rank; namely, it belongs to the class of  ``Moebius--Kantor groups'' which was introduced  in \cite[\ts 4]{rd}.) Flat planes in $\Delta$  leads  to  tessellations of the Euclidean plane by equilateral triangles. 
The definition of the Pauli puzzles relies  on a special property of the action of $G$ on $\Delta$, namely, the equivariance relation with respect to the cyclic permutation 
$T\colon X\mapsto Y\mapsto Z$ of the Pauli matrices; see \ts\ref{S - action G Delta}.

\section{Permuting the Pauli matrices; the basic construction}\label{S - triangle}

The Pauli matrices anticommute and satisfy the relations  
\begin{enumerate}
\item[]
$X^2=Y^2=Z^2=1$
\item[]
$XY=iZ$, $YZ=iX$, $ZX=iY$
\end{enumerate}
where $i=XYZ$ is central. 

The group they generate viewed viewed as unitary matrices is a finite subgroup $P:=\ip{X,Y,Z}$  of $\SU(2)$---the Pauli group---which isomorphic to a central product of the cyclic group $\ZI/4\ZI$ and the dihedral group $D_4$. 

Observe  that the Pauli matrices can be permuted cyclically
\[
 X\mapsto Y\mapsto Z
\]
by a unique automorphism of order 3 of $P$. We let $U := P\rtimes \ZI/3\ZI$ be the corresponding semi-direct product, and let $T$ denote the corresponding generator of $\ZI/3\ZI$.

The basic construction presented in this section, using triangles of groups, is that of a group $G$ which contains $U$, from which the Pauli ring puzzles are derived. In particular, the development of this complex of groups satisfies an equivariance property described in \ts \ref{S - action G Delta} which is crucial.

A triangle of groups is a commutative diagram of groups and injective homomorphisms (viewed as inclusions)

\begin{center}
\begin{tikzpicture}
    \matrix (m) [
      matrix of math nodes,
      row sep=.3cm,
      column sep=.2cm,
    ] {
  &          & & U &\\
 &			& A &  & B  \\
 &		&	&  D&&  \\
\ \ \ W\ \ \ && & C && & \ \ \ V\ \ \ \\
  };
    \path[->]        (m-2-3) edge node[left,above] {}(m-1-4);
    \path[->]        (m-2-3) edge node[left,above] {}(m-4-1);
    \path[->]        (m-2-5) edge node[left,above] {}(m-1-4);
    \path[->]        (m-2-5) edge node[left,above] {}(m-4-7);
    \path[->]        (m-4-4) edge node[left,above] {}(m-4-1);
    \path[->]        (m-4-4) edge node[left,above] {}(m-4-7);

    \path[->]        (m-3-4) edge node[left,above] {}(m-2-3);
    \path[->]        (m-3-4) edge node[left,above] {}(m-2-5);
    \path[->]        (m-3-4) edge node[left,above] {}(m-4-4);
    
\end{tikzpicture}
\end{center}
such that $A\cap B=B\cap C=C\cap A =D$, $A=U\cap W$, $B=U\cap V$, $C=V\cap W$.

The direct limit of this diagram is a group $G$, which is uniquely defined. It is generally not trivial to show that $G$ is infinite, but there are many well--known constructions for which this is the case (see e.g., \cite{MumfordA2,TitsTriangles,GromovHyp}) and general  criteria  for this to happen. A comprehensive reference for triangles and more general complexes of groups is \cite{BH}.

In this paper the groups that we consider are
\begin{enumerate}
\item
[] $U= P\rtimes \ZI/3\ZI$; $V=S_3$; $W=\ZI/6\ZI$;
\item
[] $A=\ZI/3\ZI$; $B=C=\ZI/2\ZI$; 
\item
[] $D=1$;
\end{enumerate}
with maps
\begin{enumerate}
\item
[] $A,B\to U$ given respectively by $T$ and $X$;
\item
[]$B,C\to V$ by $(1,2)$ and $(2,3)$;
\item
[]$A,C\to W$ by $2$ and $3$. 
\end{enumerate}
In this case $G$ is again infinite. 

Recall that a triangle of group is said to be developable if it is isomorphic to the complex of group associated an action of $G$ on a 2-complex $\Delta$ (see Chap.\ 3 in \cite{BH}, \ts C.2.9).  Examples are the triangles of groups of \cite{TitsTriangles}. The general criterion established in \cite{StallingsTriangles,HeafligerOrbihedra,CorsonTriangles,BH} (see also \cite{GromovHyp,BB,BS}) is that if a triangle of groups satisfies a nonpositive curvature condition, it is developable. This holds for the above triangle, as we  now prove.  

The Gersten--Stallings angles are defined as follows. If $A,B$ are subgroups of a group $U$, 
 let the shortest word in the kernel of the unique homomorphism $A*B\to U$ have length $n$ (where $n:=\infty$ if the map is injective). The angle between $A$ and $B$ in $U$ is  defined to be $2\pi/n$. For example, the angles between $B$ and $C$ and between $A$ and $C$ in $V$ and $W$ equal, respectively,  $\pi/3$ and $\pi/4$. It is proved in the aforementioned references that if in a triangle of groups the sum of the angles at $U$, $V$ and $W$ is at most $\pi$ (in which case the triangle of groups is said to be nonspherical), then it is developable. 
 
 \begin{lemma}\label{L - local development U}
 The length of the shortest word in the kernel of $A*B\to U$ is 12.
 \end{lemma}

To prove this lemma, consider the local development at $U$. The link in the development is isomorphic to a directed graph $L$ with vertex set $U/A\sqcup U/B$ and edge set corresponding to the elements of $U$ (which acts simply transitively by definition), where the initial vertex of $s\in U$ is $sB$ and its terminal vertex is $sA$. (Remark: the convention here is opposite to that of \cite{StallingsTriangles},  in order for $L$ to coincide (with our notation) with the scwol of the barycentric subdivision associated with its topological realization.) Thus, $B$, $TB$ and $T^2B$ are in the 1-neighborhood of $A$, while $A$ and $XA$ are in the the 1-neighborhood of $B$. By equivariance, it follows that $L$ is the barycentric subdivision of a topological graph $L'$, which is the undirected Cayley graph of the Pauli group with respect to the Pauli matrices.  It is easy to check that this graph has girth 6 (in fact, this graph is well--known to be isomorphic to the Moebius--Kantor graph, which has girth 6), which proves the lemma.

(It follows that  the group $G$ is an example of a Moebius--Kantor group, i.e.,  a group which acts properly discontinuously with compact quotient on (the topological realization of) a simplicial complex $\Delta$ of dimension 2, whose links at vertices are isomorphic to the Moebius--Kantor graph. These groups were introduced in \cite{rd}.)

In particular, the triangle of group is nonspherical and developable. Furthermore, if one defines a flat metric on the fundamental domain using the Gestern--Stallings angle, then $\Delta$ verifies the link condition (see \cite[Chap. II.5]{BH}) and therefore is a CAT(0) space.

We shall maintain the notation $(X,Y,Z,T,P,U,G,\Delta)$ throughout the paper.

\section{A  Pauli puzzle set}\label{S - Puzzle set}

Let $E$ denote the Euclidean plane endowed with the standard tiling by unit equilateral triangles.    
We want to build $E$ using basic tiles, which are distinguished unit triangles of three types, $X$, $Y$, and $Z$, corresponding to the Pauli matrices. 

As mentioned in the introduction, we will do so by defining a ring puzzle in the sense of \cite{autf2puzzles}, assigning conditions that hold around every vertex of $E$ and determine the local neighbourhoods. 

We require the following conditions:

\begin{enumerate}
\item
 that tiles of the same type are not adjacent; equivalently, the product of two adjacent matrices is not the identity
\item
that the cyclic  product of the matrices neighbouring  every vertex is the identity matrix
\end{enumerate} (the latter condition is independent of the chosen initial letter in the product and the orientation of $E$).  

A solution to this ring puzzle set will be called a Pauli puzzle. 

\begin{lemma}\label{L - rings}
The following words of length 6 in $\ZI/2\ZI*\ZI/2\ZI*\ZI/2\ZI$ 
\begin{enumerate}
\item
[] $XYXZYZ$
\item
[] $YZYXZX$
\item
[] $ZXZYXY$
\end{enumerate}
are a fundamental set of reduced words of length 6 representing the identity in the Pauli group.  
\end{lemma}

The proof is  straightforward. 

Furthermore, the local neighbourhoods in a Pauli puzzle are precisely described by these reduced words.

Drawn below is a lozenge of Pauli matrices satisfying the two requirements.

\begin{figure}[H]\centering
\includegraphics[width=5cm]{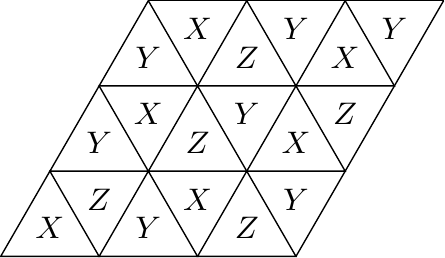}
\end{figure}

We claim:

\begin{theorem}\label{T - Pauli classifiable}
The Pauli puzzles are classifiable. 
\end{theorem}

The proof of Theorem \ref{T - Pauli classifiable} is given in \ts \ref{S - T strips} and \ts \ref{S - proof}.

We shall often make use of the following lemma, which expresses the ``nonpositive curvature'' of Pauli puzzles. The latter property is defined in \cite[\ts 2]{autf2puzzles}.

\begin{lemma}\label{L - Pauli NPC}
Every reduced word of length 4 in $\ZI/2\ZI*\ZI/2\ZI*\ZI/2\ZI$, which appears at most once as a subword of a cyclic permutation of the reduced words found in Lemma \ref{L - rings}, or their inverses.  
\end{lemma}

\begin{proof}
A reduced word of length 4 which appears contains two occurrences of a letter, which might be separated by 1 or 2 letters. It is straightforward to check  that a given word fails to appear twice.
\end{proof}

\section{The action of $G$ on $\Delta$}\label{S - action G Delta}

The action of $G$ on $\Delta$ has an equivariance property which allows us to recognize Pauli  puzzles, as defined in  \ts\ref{S - Puzzle set}, as embedded portions of  $\Delta$.  More precisely, we show that one can assign, to every face in $\Delta$ a Pauli matrix in $X$, $Y$, or $Z$, in an equivariant way, with respect to the cyclic action permutation of Pauli matrices from \ts\ref{S - triangle}. 

We view the groups $U,V,W$ defined in \ts\ref{S - triangle} as subgroups of $G$. Let $T' := (2,3)\in V$. In the group $G$, the element $T'$ commutes to $T$ because $W$ is abelian.  Since $XT=TY$ and since $G$ is generated by the three elements $X$, $T$, $T'$ (choosing a a maximal subtree; we recall that more precisely, a triangle of groups defines a presentation of the group $G$, see \cite[Chap.\ 3.C, Theorem 3.7]{BH}), this  shows that every element $g\in G$ can put in the form
\[
g = T^n g'
\]
where $n=0,1,2$, and $g'$ belongs to the subgroup $G'$ of $G$ generated by $X$, $Y$, $Z$, $T'$. We let $\tau(g)=T^n$ be the corresponding map. This is a group homomorphism onto the subgroup generated by $T$, which extends the semi-direct product decomposition of $U$ to $G$.

Let $\Delta^2$ denote the set of triangles in $\Delta$. 

\begin{theorem}\label{T - Equivariant}
There exists a unique assignment $\pi\colon \Delta^2\to \{X,Y,Z\}$ such that 
\[
\pi(g t) = \tau(g)^{-1}\pi(t)\tau(g)
\]
for every $g\in G$ and $t\in \Delta^2$, and such that $\pi(t_0)=X$, where $t_0$ is the fundamental triangle in $\Delta$. Furthermore, this assignment is invariant under the action of $V$. 
\end{theorem}

\begin{proof}
We first construct the assignment $\pi\colon \Delta^2\to \{X,Y,Z\}$. 
Let $t_0$ denote the fundamental triangle in $\Delta$. For every $t\in \Delta^2$, there exists a unique $g\in G$ such that $g(t_0)=t$ (since $D=1$). Write $g=T^ng'$ where $g'\in G'$ as above. We let
\[
\pi(t) = \begin{cases} X\text{ if } n=0\\ Y\text{ if } n=1\\ Z\text{ if } n=2
\end{cases}
\]
It is clear that $\pi(t_0)=X$, an in general, $\pi(t)=\tau(g)^{-1}X\tau(g)$. If $h\in G$, then
\[
\pi(ht)=\tau(gh)^{-1}X\tau(gh)=\tau(h)^{-1}\pi(t)\tau(h)
\] 
This proves equivariance and invariance under the action of $V$. 
\end{proof}

It follows from  $V$-invariance  that the map $\pi$ is constant on the equilateral triangles of $\Delta$. Furthermore, it is injective on the equilateral triangles adjacent to a given edge, i.e., these triangles belong to pairwise distinct matrices. Indeed, this is true of the edge stabilized by $T$ by equivariance, and therefore of every edge of $X$, since by definition $G$ acts transitively on the corresponding flags.     

Let $p$ be a Pauli puzzle and $H$ be the symmetry group $p$, i.e.,  the subgroup of  isometries of the Euclidean plane which preserves the tessellation by equilateral triangles. Fix a fundamental triangle $t$ for the action of $H$ in a face associated with the matrix $X$ in $p$.

\begin{lemma}There exists a unique embedding $f\colon p\to \Delta$ such that $f(t)=t_0$ which preserves the Pauli matrices.
\end{lemma}

\begin{proof}
We construct $f$ by induction, starting from the unique mapping $t\mapsto t_0$ from $p$ into $\Delta$ preserving the angles. Since $\pi$ is invariant under $V$, this defines a unique map $f_0 : p_0 \to \Delta$, where $p_0$ is the (equilateral) face  in $p$ is marked  $X$, such that $f_0(t)=t_0$. For $n\geq 1$ we let $p_n$ denotes the set of faces in $p$ which are adjacent to $p_{n-1}$. We assume $f_{n-1}$ constructed and let $A$ be a face in $p_n\setminus p_{n+1}$ with matrix $M\in \{X,Y,Z\}$. 

Suppose first that the intersection of $A$ and $p_{n-1}$ is an edge $e$ in $p$, this edge contains a unique face $B\in  p_{n-1}$ with matrix $N\neq M$. In $\Delta$,  the edge $f(e)$ is adjacent to three triangles belonging to pairwise distinct matrices. By induction, $f(B)$ is marked $N$ and thus there exists a unique extension of $f_{n-1}$ to $A$ with is consistent on the marking. 

Suppose now that the intersection of $A$ with $p_{n-1}$ is a vertex $v$. In the ring $r$ at $v$, the extension of $f_{n-1}$ determine a path $r'\inj r$ of length 4 or 5 which, by Lemma \ref{L - Pauli NPC},   can be embedded in a ring a unique way. Furthermore, $f$ takes $r'$ to path in the link of $f(v)$. In the local development at $f(v)$ in $\Delta$ (see Lemma \ref{L - local development U}), this path extends in a unique way in the undirected Cayley graph of the Pauli group with respect to the Pauli matrices. This defines a unique extension of $f_{n-1}$ to a map $f_n$ from $p_n$ to $\Delta$. We set $f:=\varinjlim f_n$.   
\end{proof}

Since conversely, an embedding $E\to \Delta$ determines a puzzle, where the Pauli matrices are given by the restriction $\pi_{|E}$ to the image of $E$; we obtain the following result.

\begin{theorem}\label{T - 1-1- correspondence}
There is a 1-1 correspondence between the set of Pauli puzzles pointed at a fundamental triangle $t$ (for the action of $H$), and the set of simplicial embeddings of (the barycentric subdivision of) $E$ into $\Delta$ taking $t$ to the triangle $t_0$.
\end{theorem}

The paper \cite{autf2puzzles} introduced ring puzzles as a mean to prove the so-called ``flat closing conjecture'' for certain groups for which the classification theorem is available. Combining Theorem \ref{T - Pauli classification 2} with the techniques of  \cite{autf2puzzles}, it is not difficult to give an abstract proof that $G$ contains an copy of $\ZI^2$.
This can, of course, be proved directly. A way to do it is to use Theorem \ref{T - 1-1- correspondence}. The group $H$ contains torsion free subgroups of finite index generated by two translations in distinct directions. Consider for example the following fundamental domain:
\begin{figure}[H]
\includegraphics[width=4cm]{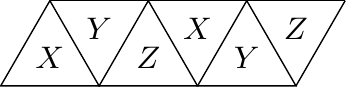}
\end{figure}

If $H_0\subset H$ is a subgroup isomorphic to $\ZI^2$, an $H_0$-periodic puzzle is a pauli Puzzle whose Pauli matrices are invariant under the action of $H_0$; for such a puzzle, $H_0$ embeds into $G$ (conversely, a subgroup $H_0$ isomorphic to $\ZI^2$ in $G$ defines a periodic $H_0$ puzzle in $G$, which is a consequence of the flat torus theorem \cite[II.7]{BH} and show that the correspondence in Theorem \ref{T - 1-1- correspondence} takes periodic puzzles to periodic planes). It is easy to see that the fundamental domain shown above extends into an $H_0$-periodic puzzle, where $H_0$ is the group generated by $(3,0)$ and $(0,1)$ in the obvious basis.

\section{$T$-strips}\label{S - T strips}

Let $p$ be Pauli puzzle and  $s$ be a 1-strip, which can be finite, semi-infinite, or bi-infinite, in $p$.

 We call $s$ a \emph{$T$-strip} if if it contains at least three triangles and every triple of consecutive triangles  contains the three Pauli matrices. Thus, the sequence of matrices along a  $T$-strip are obtained by applying $T$ in one direction and  $T^2$ in the other.

In order to prove Theorem \ref{T - Pauli classifiable} we  establish in this section a series of lemmas on $T$-strips.

 A \emph{maximal semi-infinite $T$-strip} a semi-infinite 1-strip $s$ which is not a $T$-strip,  such that $s\setminus t$ is, where $t$ is the initial triangle of $s$. 
  Here and below, we use the following convention on the difference set for 1-strips: if $s$ is a 1-strip and $t$ is a 1-strip included in  $s$, then $s\setminus t$ denotes the closure of the set theoretic difference of $s$ and $t$;  this is the union of all  (closed) triangles in $s$ which are not in $t$. Thus, $s\setminus t$ is either a 1-strip, or a union of two disjoint 1-strips. If $t$ is the initial triangle of $s$, then $s\setminus t$ is itself a 1-strip.

\begin{lemma}\label{L - semi infinite}
A maximal semi-infinite $T$-strip $s$ admits a unique extension into a sector  of $p$ which is the convex hull of $s$ and a semi-infinite $T$-strip $s'$   intersecting $s$ with an  angle value $\pi/3$.      
\end{lemma}

\begin{proof}
Since $s$ is maximal, the matrix on the initial triangle $t$ coincides with that on the triangle $t'$  of $s$ such that  $t\cap t'=\{v\}$. By Lemma \ref{L - rings}, $t$ and extend into a unique hexagon of $p$ centered at $v$, which, by Lemma \ref{L - Pauli NPC}, extends into a unique strip $s_1$ parallel to $s_0:=s$. It follows that $s_1$ is again a maximal semi-infinite $T$-strip. By induction, we define in this way a sequence of parallel strips $s_n$,  and a sector, the convex hull of $\bigcup _{n\geq 0} s_n$, whose boundary 1-strip $s'\neq s$ forms an angle $\pi/3$ with $s$. It is easy to check that $s'$ is again a $T$-strip.
 \end{proof}

A \emph{maximal  finite} $T$-strip is a  finite 1-strip $s$ such that $s\setminus t$ and $s\setminus t'$ are not $T$-strip, but $s_{-1}:=s\setminus t\vee t'$ is, where $t$ (resp.\ $t'$) is the initial (resp.\ final) triangle of $s$.

\begin{lemma}\label{L - reflection T strip}
Suppose $I$ is a simplicial segment in $p$ of length at least 3 and $s$ is a $T$-strip on $I$. Let $\overline s$ be the reflection of $s$ in $p$ along $I$. Then $\overline s_{-1}$ is a $T$-strip. 
\end{lemma}

\begin{proof}
By Lemma \ref{L - rings}, since this is true for $s$, the three triangles in $\overline s$ containing the first two edges in $I$ are associated with distinct Pauli matrices. By Lemma \ref{L - Pauli NPC}, the extension of $\overline s_{-1}$ is uniquely determined by $s$ and these three matrices. It is then straightforward to check that the extension is given by $T$ or $T^2$, depending on which transformation $s$ is associated with; in fact, it must be the same transformation for both strips.
\end{proof}

\begin{lemma} Suppose $s$ is a maximal finite $T$-strip. Then the bases of $t$ and $t'$ belong to opposite sides of $s$.
\end{lemma}

\begin{proof}
Suppose $t$ and $t'$ are on the same side $I$ of $s$.  Let $\overline s$ denote the reflection of $s$ with respect to $I$ in $p$.  Since  $s_{-1}$ is a $T$-strip, so is $\overline s_{-2}:=(\overline s_{-1})_{-1}$. Furthermore, the matrix sequence on $\overline s_{-2}$ is given by applying $T$ or $T^2$ depending on the transformation in $s_{-1}$. This implies that the matrix associated with a triangle in $\overline s_{-2}$, sitting on an edge $e$ of $I$ from $t$ to $t'$, is the same as that associated with the triangle in $s_{-1}$ containing the extremity of $e$. In particular, the matrix associated with the triangle in $\overline s_{-2}$ sitting on the last edge $e$ coincides with that associated with the triangle in $s$ containing the extremity of $e$. Now, if $s\setminus t$ is not a $T$-strip, but $s_{-1}$ is, then $t'$ and the triangle in $s$ sitting on the edge $e$ share the same Pauli matrix.  This contradicts Lemma \ref{L - rings}, more precisely, the fact the reduced words of the form $(MN)^2$ in $\ZI/2\ZI*\ZI/2\ZI*\ZI/2\ZI$, where $(M,N)$ is a pair of distinct Pauli matrices, are not subwords of the words described in this lemma.
\end{proof}

The following is an obvious consequence.

\begin{lemma}
A maximal finite $T$-strip is isometric to a parallelogram, and contains an even number $n$ of triangles. 
\end{lemma} 

In addition, one can describe the transverse structure.

\begin{lemma}\label{L - finite T strip} Suppose that $s$ is maximal finite $T$-strip. Then $s$ admits a unique extension into a bi-infinite strip transverse to $s$, of width $\frac n 2+1$, which is the union of that many parallel bi-infinite $T$-strips.
\end{lemma}

\begin{proof} By a previous lemma, $t$ and $t'$ belong to opposite sides of $s$.
An argument similar to that of Lemma  \ref{L - semi infinite} shows that there exists a unique extension $s'\supset s$ such that $s=s'_{-1}$, such that the bi-infinite strip of width  $\frac n 2+1$ formed of translates extends $s$ in a unique way transversally on both sides. The matrices associated with these translates are shifted by $T$ or $T^2$, and threfore are  periodic of period 3. It follows that the bi-infinite strips in the transverse direction are $T$-strips.    
\end{proof}

\section{Proof of Theorem \ref{T - Pauli classifiable}}\label{S - proof}

Let $p$ be a Pauli puzzle. Suppose that $p$ contains a bi-infinite $T$-strip $s$. Let $I$ be a segment included in the boundary of $s$. By Lemma \ref{L - reflection T strip}, the reflection $\overline{s_{I}}_{-1}$ is a $T$-strip. Since the length of $I$ is arbitrary, so is $\overline s$. It follows by induction that $p$ is a union of $T$-strips.

We suppose, for the remainder of this section, that $p$ does not contain a bi-infinite $T$-strip. 

Observe that by applying Lemma \ref{L - rings} and Lemma \ref{L - finite T strip}, the Pauli puzzle $p$ contains a maximal semi-infinite $T$-strip.  Indeed, every triangle $t$ belongs to $T$-strip $s_0$; let $s$ be a maximal $T$-strip containing $s_0$. Then $s$ is semi-infinite: if it were finite, then Lemma \ref{L - finite T strip} would imply that $t$ belongs to a bi-infinite $T$-strip transverse to $s$.

\begin{lemma}\label{L - excluding a Pauli matrix}
Suppose $t_0$ is a triangle, $s$ is a maximal semi-infinite $T$-strip containing $t_0$,  $r$ is the bi-infinite 1-strip in $p$ containing $s$. Then there exists an integer $n\geq 3$, a maximal semi-infinite $T$-strip $s'\neq s$ included in $r$, such that $r\setminus (s_{-3}\cup s'_{-3})$ is a 1-strip of length $n$ which excludes one of the Pauli matrices, where for a semi-infinite strip $s$, we use the notation $s_{-k}$ to refer to $s\setminus (t_1\cup  \cdots \cup t_k)$, where the $t_i$'s are the $k$ initial triangles in $s$. 
\end{lemma}

\begin{proof}
Since $s$ is maximal, there exists a Pauli matrix, say $M$, which is excluded from the first three initial triangles. Let $b$ be the largest 1-strip in $r$ which contains these three triangles and excludes $M$. 

We claim that $b$ is finite. 

Suppose not. Let $b_1$ be a 1-strip parellel to $b$ in $p$ with a common initial vertex $v$. For every vertex $v'\neq v$ common to $b$ and $b_1$, it follows from Lemma \ref{L - rings} that the triangles in $b_1$ containing $v'$ exclude a Pauli matrix $N$ (necessarily such that $N\neq M$). By connectedness of $b_1$, the triangles adjacent to every vertex $v'\neq v$ exclude the same matrix $N$. 

Let $r_1$ be the 1-strip in $p$ containing $b_1$, and let $b_1'$ be the maximal 1-strip included in $r_1$ which exclude matrix $N$. Then by Lemma \ref{L - reflection T strip} the 1-strip  $r_1\setminus b_1'$ is a $T$-strip, being adjacent to $s$.

Furthermore, Lemma \ref{L - rings} at the unique vertex in $r\cap b_1'\cap (r_1\setminus b_1')$ shows that the angle between the initial edges of the two strips $b$ and $b_1'$ is $\pi$. Therefore, by induction setting $b_0=b$, we define a sequence $(b_k)_{k\in \ZI}$ of 1-strips, which are consecutively adjacent, which individually exclude a Pauli matrix (which depend on the given strip), and whose union $P:=\bigcup_{k\in \ZI} b_k$ is a half plane contained in $p$. This however implies that $P$ contains a bi-infinite $T$-strip which is transverse to the $b_k$'s, contrary to our assumption that $p$ doesn't. 

This proves that $b$ is finite. 

We now prove that $r\setminus b$ is a union of two $T$-strips. Since $b$ contains at least 3 triangles, and these two strips are maximal $T$-strips, this proves the lemma.

Suppose that $r\setminus (s\cup b)$ fails to be a $T$-strip. Then it contains a 1-strip $b'$ which excludes a Pauli matrix. Since $b$ is maximal, $b$ and $b'$ are disjoint. Suppose that $b'$ is the 1-strip in $r\setminus (s\cup b)$ closest to $b$ satisfying this property. Then the 1-strip between $b$ and $b'$ in $r$ a $T$-strip, which is clearly maximal and finite. By Lemma \ref{L - finite T strip}, $p$ must contain a bi-infinite $T$-strip which is transverse to $r$.  This again contradicts our assumption and concludes the proof.
\end{proof}

\begin{lemma}
In the notation of Lemma \ref{L - excluding a Pauli matrix}, the length of $r\setminus (s_{-3}\cup s'_{-3})$ is odd.
\end{lemma}

\begin{proof}
Suppose towards a contradiction that $n=|b|$ is even, $b:=r\setminus (s_{-3}\cup s'_{-3})$. Then $n\geq 4$, and by Lemma \ref{L - rings}, there there are two parallel adjacent $T$-strips of length $\geq 4$ which are transverse to $b$. Let $s$ and $s'$ be maximal $T$-strips containing them. Since $s$ and $s'$ are parallel and adjacent, and by assumption they are not bi-infinite, Lemma \ref{L - reflection T strip} shows that they are asymptotic in $p$. 
In fact, one can map $s_{-1}$ to $s'_{-1}$ with a translation of $E$ that preserves the Pauli matrices. Let $v$ be the vertex on the boundary of $s_{-1}\cup s'_{-1}$ (a union of two $T$-strips) which belongs to $s\cap s'$. The three triangles in $s_{-1}\cup s'_{-1}$ adjacent to $v$ contain two Pauli matrices, say $M$ and $N$, with $M$ say, appearing once. However, by maximality of $s$ and $s'$, the matrix $M$ extends $s_{-1}$ and $s_{-1}'$ into $s$ and $s'$, respectively. Thus, it appears three times in the vicinity of $v$, which is a contradiction.  
\end{proof}

We also record the following, whose proof is contained in that of  Lemma \ref{L - excluding a Pauli matrix}.

\begin{lemma}\label{L - b is finite}
There does not exist a semi-infinite 1-strip $b$ which excludes a given Pauli matrix.  
\end{lemma}

Furthermore, the following holds:

\begin{lemma}
Every strip in $p$ is of the form described in Lemma \ref{L - excluding a Pauli matrix}.
\end{lemma}

\begin{proof}
Let $r$ be a 1-strip in $p$, $t_0$ a triangle of $r$, $b$ be the 1-strip in $r$ which contains $r$, and is maximal with respect to the property of either being $T$-strip, or excluding a Pauli matrix. In the first case, \ref{L - excluding a Pauli matrix} applies; in the second, $|b|$ is finite by Lemma \ref{L - b is finite}; let $t_1$ be the first triangle in a component of $r\setminus b$, and let $s$ be the maximal $T$-strip containing $t_1$. Then $s$ is a semi-infinite strip by lemma \ref{L - finite T strip}, so Lemma \ref{L - excluding a Pauli matrix} applies in this case as well. 
\end{proof}

There always exist 1-strips in $p$ with arbitrary large $b$'s in them, as we shall see below. On the other hand, the minimum length of $b$ is  attained:

\begin{lemma}
There exists a 1-strip in $p$ realizing the minimum $n=3$ in Lemma \ref{L - excluding a Pauli matrix}.
\end{lemma}

\begin{proof}
Let $m$ denote the minimum over the set of 1-strips in $p$, $r$ be a 1-strip in $p$ realizing the minimum,  $b:=r\setminus (s_{-3}\cup s'_{-3})$, with $m=|b|$. By the previous lemmas, we may suppose  that $m$ is an odd number $\geq 5$. Let $I$ be the short side in $b$, of length $(m-1)/2$. The latter integer is $\geq 2$ and Lemma \ref{L - rings} applies to show that the the 1-strip $r'$ parallel to $r$ and containing $I$ containing a 1-strip excluding the Pauli matrix in $r$ belonging to the triangles adjacent to $I$, of length $m-2$, contradicting the minimality of $m$. 
\end{proof}

\begin{lemma}\label{L - 3 strips}
There exists pairwise transverse 1-strips $r,r',r''$ realizing the minimum $n=3$ in Lemma \ref{L - excluding a Pauli matrix} whose convex closure is $p$ and common intersection an equilateral triangle in $p$.
\end{lemma}

\begin{proof}
Let $r$ be such a strip. We maintain the notation $b$ introduced above (so $|b|=3$). Let $e$ be the longitudinal boundary of length 1 in $b$ and $t$ be the triangle of $p$ not contained in $b$ and containing $e$. We distinguish two cases. Let $M$ be denotes the Pauli matrix which appears twice in $b$, and let $N$ be the Pauli matrix which is excluded. 

Suppose that $N$ is the Pauli matrix which belongs to $t$. Let $r'$ and $r''$ be the 1-strips containing the edges of $t$ distinct from $e$ and not containing $t$. Let $n'=|b'|$ and $n''=|b''|$ be the corresponding parameters for $r'$ and $r''$.  We claim that $n'=n''=3$. Let $v$ and $v'$ be the extremities of $e$. Since $b$ is maximal, the matrix $N$ belongs to the two triangles of $r$ respectively containing $v$ and $v'$ and not contained in $b$. Applying Lemma \ref{L - rings} to the 3 vertices of $t$, we see that  $b'$ and $b''$ are maximal of length 3 and contain the same matrices as $b$. It follows immediately that the three 1-strips in $p$ containing $t$ also have the same parameter $n=3$, whose corresponding 1-strips of length 3 excludes the matrix $M$, extending into $p$ with two maximal semi-infinite $T$-strips.  This proves the lemma if $N$ belongs to $t$.

Suppose now that $M$ is the Pauli matrix which belongs to $t$.  Let $\overline t$ be the triangle of $r$ containing $e$, and let $r'$, $r''$ be the two 1-strips $\neq r$ containing $\overline t$. Applying Lemma \ref{L - rings} to the 3 vertices of $\overline t$, we see that $r,r',r''$ have parameter $n=3$, and extend into $p$ with two maximal semi-infinite $T$-strips. This conclude the  proof of the lemma.
\end{proof}

This concludes the proof of Theorem \ref{T - Pauli classifiable}:  by lemma \ref{L - Pauli NPC}, the previous result determines, in the case that $p$ does not contain a bi-infinite $T$-strip, a unique Pauli puzzle, with a prescribed Pauli matrix on the intersection $t$ of the three 1-strips constructed in Lemma \ref{L - 3 strips}. If $M$ is the Pauli matrix, we shall call this puzzle the  Pauli $M$-puzzle. 

We have established a more precise result which we formulate now. 

\begin{theorem}\label{T - Pauli classification 2}
A Pauli puzzle is either a Pauli $M$-puzzle, where $M\in \{X,Y,Z\}$ is a Pauli matrix, or it is a union of parallel $T$-strips.
\end{theorem}

Finally, one can describe the transverse structure, in the case of a union of parallel $T$-strips, by a sequence, either finite, semi-infinite, or bi-infinite, of integers (possibly infinite) parametrizing the structure of $p$ transverse to the $T$-strips; this sequence is then a complete isomorphism invariant of ring puzzles in this case. We shall leave this as an exercise.

Let us conclude this section by noting that the Pauli  puzzle construction can be generalized in quite a few ways. If $S$ is a set of involutions in a group $G$, one can define a ring puzzle set with $S$-tiles having as rings the set of reduced words of length 6 in the free product of $\ZI/2\ZI$ over $S$ whose product equals the identity in $G$. The  conditions for a tiling to be a ring puzzle in this sense can be thought of as ``homological'' (although not in a usual sense) in nature. If $(I,J)$ denotes the set of maps from $I$ to $J$, then we have a ``boundary operator'' $d$ defined by
\[
d\colon (E^2,S)\to (E^0,T)
\]
where $E^r$ denotes the $r$-skeleton of $E$ and $T$ the set of conjugacy classes of $G$, taking $p$ to  the cyclic product
\[
dp(v) := {\prod_{t\ni v}}^{\circlearrowleft} p(t)
\]
and $p$ satisfies (2) if and only if $dp=1$, where the cyclic product is independent of the chosen initial letter and the orientation of $E$. This defines a set of ``cycles'', let us say  $Z_{2\to 0}^{\circlearrowleft}(E,S)$. Condition (1) removes a certain set of ``subcycles'' from   $Z_{2\to 0}^{\circlearrowleft}(E,S)$. These are elements $p$ of $Z_{2\to 1}^{\circlearrowleft}(E,S)$  for which there exists $v\in E^0$ for which a strict subproduct of ${\prod_{t\ni v}}^{\circlearrowleft} p(t)$ is trivial. With these definitions, the corresponding set  of puzzles is 
\[
H_{2}^\circlearrowleft(E,S):=Z_{2\to 0}^{\circlearrowleft}(E,S)\setminus Z_{2\to 1}^{\circlearrowleft}(E,S).
\] 
For example, if $S=\{(1,2),(2,3)\}$, then $H_2^ \circlearrowleft(E,S)$ is the checker tiling of $E$; Theorem \ref{T - Pauli classifiable} computes $H_2^\circlearrowleft(E,S)$ when $S$ is the set of Pauli matrices. It seems also interesting to study the case where $E$ is a regular tiling of the sphere or hyperbolic space.

\section{Pauli puzzles as ring puzzles}

We conclude with a  presentation of the Pauli $X$-puzzles as ring puzzle in the sense of \cite{autf2puzzles}. 

We use the following encoding of the Pauli matrices:

\begin{enumerate}
\item
[]
\begin{tikzpicture}
\node at (0,0) {$X=\left(\begin{smallmatrix}0&1\\1&0\end{smallmatrix}\right)=$};
\shade[ball color=red!80!yellow] (1.1,0) circle (0.5ex);
\end{tikzpicture}
\item
[]
\begin{tikzpicture}
\node at (0,0) {$Y=\left(\begin{smallmatrix}0&-i\\i&0\end{smallmatrix}\right)=$};
\shade[ball color=blue!60!green] (1.1,0) circle (0.5ex);
\end{tikzpicture}
\item
[]
\begin{tikzpicture}
\node at (0,0) {$Z=\left(\begin{smallmatrix}1&0\\0&-1\end{smallmatrix}\right)=$};
\shade[ball color=yellow!70!orange] (1.1,0) circle (0.5ex);
\end{tikzpicture}

\end{enumerate}

To define a ring puzzle one needs two sets, a set of shapes and a set of rings. These sets are:

\begin{enumerate}
\item
[]  shape set:
\bigskip
\[
\includegraphics[width=5cm,trim = 0cm .3cm 0cm 0cm]{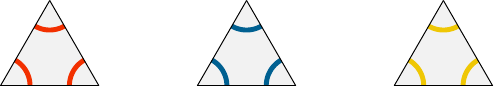}
\]
\vskip1cm  
\item
[] ring set:
\bigskip
\[
\includegraphics[width=1cm,trim = 0cm .9cm 0cm 0cm]{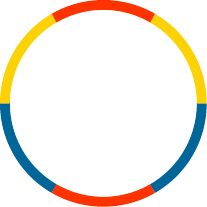}\ \ \ \ \ \ \ \ \includegraphics[width=1cm,trim = 0cm .9cm 0cm 0cm]{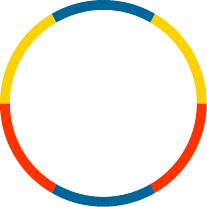}\ \ \ \ \ \ \ \ \includegraphics[width=1cm,trim = 0cm .9cm 0cm 0cm]{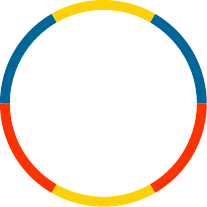}
\]
\vskip1cm  

\end{enumerate}

It follows from the results in the preceding sections that there exist precisely two Pauli $X$-puzzles. 

\newpage

\begin{figure}[H]
\includegraphics[width=13cm]{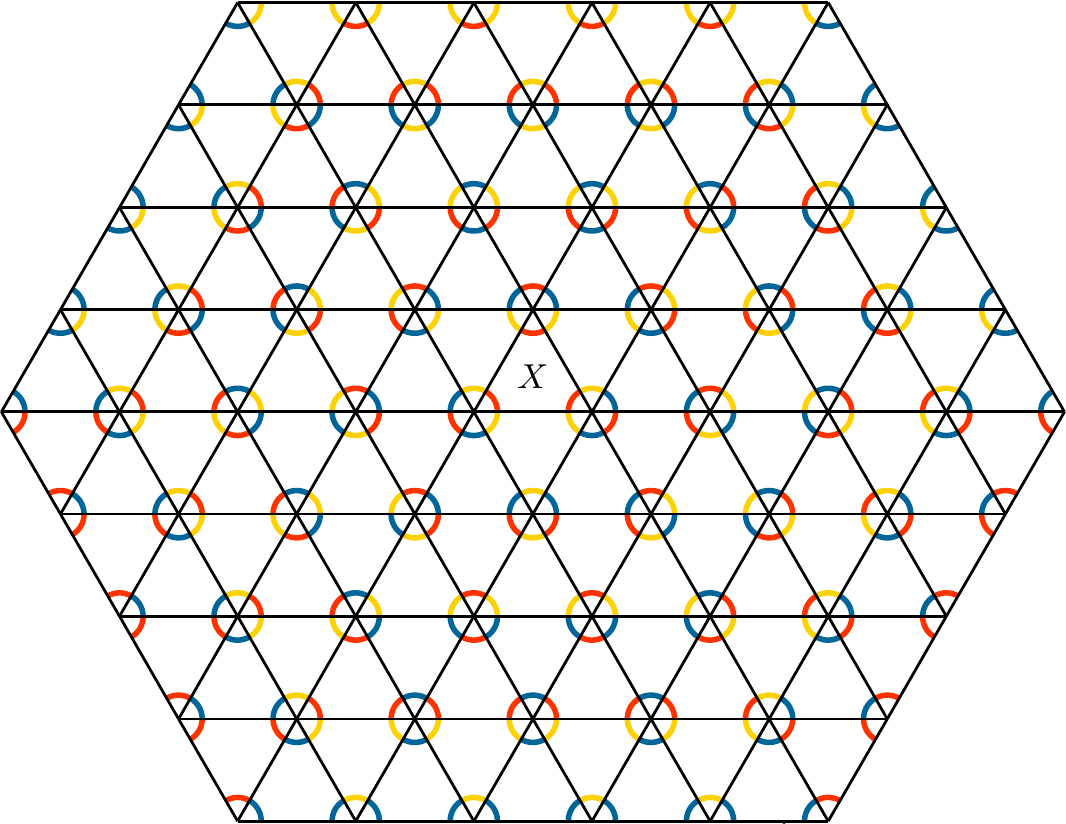}
\end{figure}

\begin{figure}[H]
\includegraphics[width=13cm]{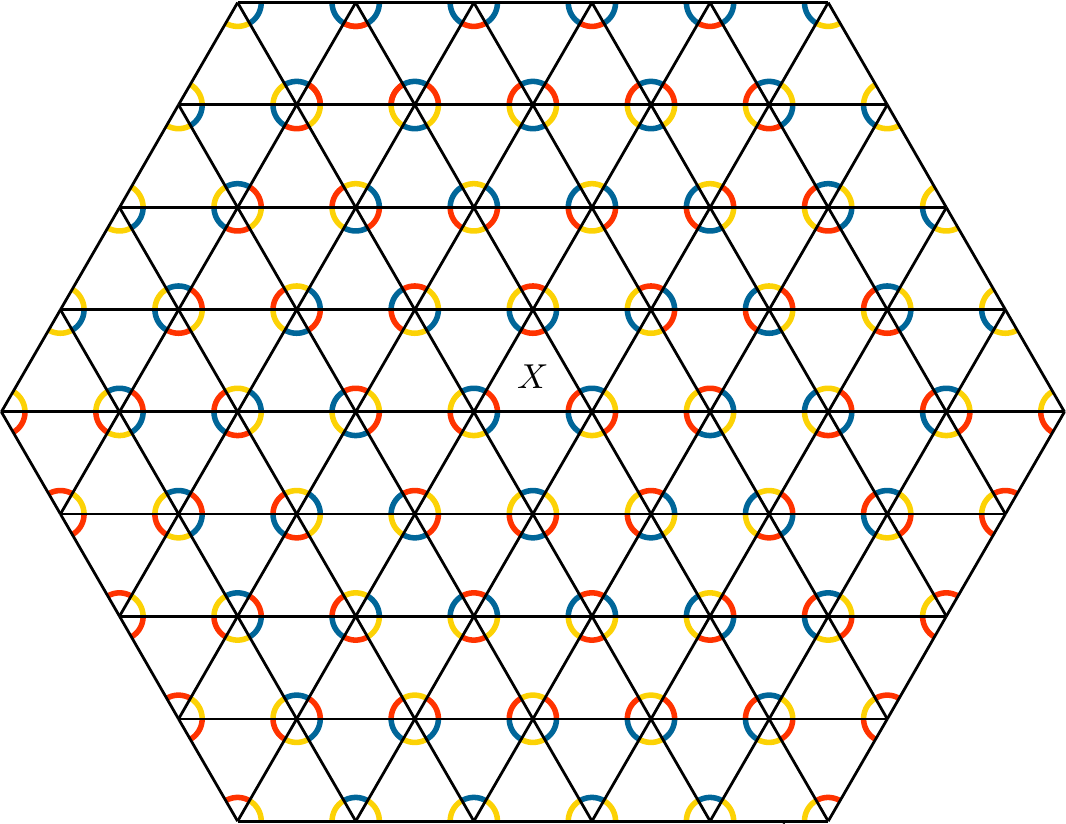}
\end{figure}

\end{document}